\documentclass{amsart}
\usepackage{amsfonts}
\usepackage{amssymb}
\usepackage{amsthm}
\usepackage{tikz}

\theoremstyle{plain}
\newtheorem{theorem}{Theorem}[section]
\newtheorem{lemma}[theorem]{Lemma}
\newtheorem{corollary}[theorem]{Corollary}

\theoremstyle{definition}
\newtheorem*{definition}{Definition}

\newtheorem*{example}{Example}
\newtheorem*{problem}{Problem}

\theoremstyle{remark}
\newtheorem*{remark}{Remark}

\newcommand{\gen}[1]{\langle#1\rangle}

\newcommand{\Z}{\mathbb{Z}}

\newcommand{\Aut}{\mathrm{Aut}}

\newcommand{\Cay}{\mathrm{Cay}}
\newcommand{\G}{\Gamma}
\newcommand{\diam}{\mathrm{diam}}
\newcommand{\D}{\mathcal{D}}

\begin{document}
\title[Relative Cayley graphs]{Relative Cayley graphs of finite groups}

\author{M. Farrokhi D. G.}
\address{Mathematical Science Research Unit, College of Liberal Arts, Muroran Institute of Technology, 27-1, Mizumoto, Muroran 050-8585, Hokkaido, Japan.}
\email{m.farrokhi.d.g@gmail.com}

\author{M. Rajabian}
\address{Ferdowsi University of Mashhad, International Branch, Mashhad, Iran.}
\email{mehdi.rajabian@yahoo.com}

\author{A. Erfanian}
\address{Department of Pure Mathematics, Ferdowsi University of Mashhad, Mashhad, Iran.}
\email{erfanian@math.um.ac.ir}

\date{}

\begin{abstract}
The relative Cayley graph of a group $G$ with respect to its proper subgroup $H$, is a graph whose  vertices are elements of $G$ and two vertices $h\in H$ and $g\in G$ are adjacent if $g=hc$ for some $c\in C$, where $C$ is an inversed-closed subset of $G$. We study the relative Cayley graphs and, among other results, we discuss on their connectivity and forbidden structures, and compute some of  their important numerical invariants.
\end{abstract}
\maketitle
\section{Introduction}
Cayley graphs was introduced by Arthur Cayley \cite{ac} in 1878 to give a geometrical representation of groups by means of a set of generators. This translates groups into geometrical objects which can be studied form the geometrical view. In particular, it provides a rich source of highly symmetric graphs, known as transitive graphs, which plays a central role in many graph theoretical problems as well as group theoretical problems, like expanders, width of groups, representation of interconnection networks, Hamiltonian paths and cycles that naturally arise in computer science and etc.

We intent to introduce and study special subgraphs of the Cayley graphs of a group $G$ with respect to a given proper subgroup $H$ of $G$, called the relative Cayley graphs. The \textit{relative Cayley graph} of $G$ with respect to $H$, denoted by $\G=\Cay(G,H,C)$, is a graph whose vertices are elements of $G$ such that two vertices $x$ and $y$ are adjacent if $x$ or $y$ belongs to $H$ and $x^{-1}y\in C$ for some inversed closed subset $C$ of $G\setminus\{1\}$. Clearly, $\G$ has an induce subgraph $\G'=\Cay(H,H\cap C)$, which is itself a Cayley graph. Relative Cayley graphs with respect to specific subgroups of a group $G$ provides a good source of subgraphs in covering the whole Cayley graph of $G$. Also, as we shall see in section 3, they give a criterion for a group to be an ABA-group.

In this paper, we will investigates some combinatorial and structural properties of relative Cayley graphs. In section 2, we shall obtain preliminary results on valencies and regularity of relative Cayley graphs. In section 3, we pay attention to the connectivity and diameter problems on relative Cayley graphs. We give necessary and sufficient conditions for a relative Cayley to be connected and obtain sharp upper bounds for its diameter. In section 4, the numerical invariants of relative Cayley graphs will be considered. We will determine the explicit value of independence number, dominating number, edge independence number and edge covering number as well as edge chromatic number of relative Cayley graphs. Also, we obtain lower and upper bounds for the clique number and an upper bound for the chromatic number of relative Cayley graphs. Finally, in section 5, we shall study the absence of special subgraphs in the relative Cayley graphs, which enable us to obtain a classification of all trees isomorphic to a relative Cayley graph.

Throughout this paper $G$ stands for a finite group, $H$ stands for a proper subgroup of $G$ and $C$ denotes an inversed closed subset of $G\setminus\{1\}$, that is, $C\subseteq G\setminus\{1\}$ and $C^{-1}\subseteq C$. Also, $C^*$ denotes the set $C\cup\{1\}$.
\section{Valencies}
As it is mentioned in the introduction, $\G=\Cay(G,H,C)$ is a relative Cayley graph of $G$ with respect to $H$ and $\G'=\Cay(H,H\cap C)$ is a Cayley graph of $H$ that is an induced subgraph of $\G$. We begin with identifying the neighbor of an arbitrary vertex in the relative Cayley graphs.
\begin{lemma}\label{degrees}
Let $x\in G$.
\begin{itemize}
\item[(i)]If $x\in H$, then $N_\G(x)=xC$ and $\deg_\G(x)=|C|$. Also, $N_{\G'}(x)=x(H\cap C)$ and $\deg_{\G'}(x)=|H\cap C|$.
\item[(ii)]If $x\in G\setminus H$, then $N_\G(x)=H\cap xC$ and $\deg_G(x)=|H\cap xC|=|x^{-1}H\cap C|$.
\end{itemize}
In particular, $\deg(x)=\deg(y)$ whenever $x,y$ belong to the same left coset of $H$
\end{lemma}

Let $\D(\G)$ be the set of all different valencies of vertices of $\G$. The following theorem gives a sharp upper bound for $\D(\G)$.
\begin{theorem}
We have
\[|\D(\G)|\leq\min\{[G:H],|H|+2\}\leq\lfloor\sqrt{|G|+1}\rfloor+1.\]
\end{theorem}
\begin{proof}
First we prove the first inequality. By Lemma \ref{degrees}, $|\D(\G)|\leq[G:H]$. Again, by Lemma \ref{degrees}(ii), $\deg_\G(x)=|H\cap xC|\leq|H|$  for all $x\in G\setminus H$, that is, $\deg_\G(x)\in\{0,1,\ldots,|H|\}$. Hence, $|\D(\G)|\leq|H|+2$ from which it follows that 
\[|\D(\G)|\leq\min\{[G:H],|H|+2\}.\]
Now, we prove the second inequality. If $|\D(\G)|=[G:H]$, then
\begin{flalign*}
&&&&&&&&&&&&[G:H]&\leq|H|+2&&&&&&&&&&\\
&&&&&&&&&&\Rightarrow&&|G|&\leq|H|^2+2|H|<(|H|+1)^2&&&&&&&&&&\\
&&&&&&&&&&\Rightarrow&&\sqrt{|G|}&\leq|H|&&&&&&&&&&\\
&&&&&&&&&&\Rightarrow&&[G:H]&\leq\sqrt{|G|}&&&&&&&&&&
\end{flalign*}
and if $|\D(\G)|=|H|+2$, then
\begin{flalign*}
&&&&&&&&&&&&|H|+2&\leq[G:H]&&&&&&&&&&\\
&&&&&&&&&&\Rightarrow&&|H|(|H|+2)&\leq|G|&&&&&&&&&&\\
&&&&&&&&&&\Rightarrow&&|H|^2+2|H|+1&\leq|G|+1&&&&&&&&&&\\
&&&&&&&&&&\Rightarrow&&|H|+1&\leq\sqrt{|G|+1}&&&&&&&&&&\\
&&&&&&&&&&\Rightarrow&&|H|+2&\leq\sqrt{|G|+1}+1,&&&&&&&&&&
\end{flalign*}
from which it follows that $\min\{[G:H],|H|+2\}\leq\sqrt{|G|+1}+1$, as required.
\end{proof}
\begin{example}
Let $G=\Z_2^{2n}$, $H$ be a subgroup of $G$ order $2^n$ and $\{1,g_1,\ldots,g_{2^n-1}\}$ be a left transversal to $H$ in $G$. Let $C_i$ be a subset of $g_iH$ of size $i$ for $1\leq i\leq 2^n-1$. Then $\D(\G)=2^n=\sqrt{|G|}$.
\end{example}
\begin{example}
Let $G=H\rtimes A$, where $A$ is an elementary abelian $2$-group of automorphisms of a group $H$. Let $A=\{1,\alpha_1,\alpha_2,\ldots,\alpha_{n-1}\}$. Since $\alpha_iH$ is inverse closed containing an involution $\alpha_i$, there exists a subset $C_i$ of $\alpha_iH$ such that $|C_i|=\min\{i,|H|\}$ for all  $1\leq i\leq n-2$ and
\[|C_{n-1}|=\begin{cases}n-1,&n-1\leq|H|,\\0,&n-1>|H|.\end{cases}\]
Let $C=\bigcup C_i$. Then $|\D(\G)|=\min\{[G:H],|H|+2\}$.
\end{example}
\begin{remark}
Jamali \cite{arj} and Moghaddam, Farrokhi and Safa \cite{mrrm-mfdg-hs} constructed two family of non-isomorphic $2$-groups $\{H_n\}_{n\geq3}$ and $\{K_n\}_{n\geq3}$, respectively, such that $|H_n|=|K_n|=4^n$ and $\Aut(H_n)\cong\Aut(K_n)\cong\Z_2^{n^2}$ for all $n\geq3$. Hence, by choosing subgroups $A_n$ and $B_n$ of $\Aut(H_n)$ and $\Aut(K_n)$ of order $4^n$, respectively, and defining $G_n=H_n\rtimes A_n$ or $K_n\rtimes B_n$, and $C$ as in the above example, it follows that $|\D(\G)|=\sqrt{|G|}$.
\end{remark}

In the sequel, we shall study the graph $\G$ when $\D(\G)$ is small, that is, $\G$ is a regular or a semi-regular graph. Recall that a graph is \textit{regular} if all its valencies are equal and a graph is \textit{semi-regular} if its vertices have just two possible distinct valencies. First we count the number of edges of $\G$.
\begin{lemma}\label{edgesnumber}
$|E(\G)|=|H|(2|C|-|H\cap C|)/2$.
\end{lemma}
\begin{proof}
By Lemma \ref{degrees}, $|E(\G)|=|H||C\setminus H|+|H||H\cap C|/2$, from which the result follows.
\end{proof}
\begin{theorem}\label{regular}
$\G$ is regular if and only if $[G:H]=2$ and $H\cap C=\emptyset$ if and only if $\G$ is a Cayley graph.
\end{theorem}
\begin{proof}
If $\G$ is regular, then by Lemmas \ref{degrees} and \ref{edgesnumber}, 
\[|H|(2|C|-|H\cap C|)=2|E(\G)|=|G||C|=|H|[G:H]|C|.\]
Thus $2|C|-|H\cap C|=[G:H]|C|$, which is possible only if $H\cap C=\emptyset$ and $[G:H]=2$. The converse is obvious by Lemmas \ref{degrees}. The other part is straightforward.
\end{proof}
\begin{lemma}
Let $x\in G\setminus H$. Then $\deg_\G(x)=|C|$ if and only if $x\in\bigcap_{c\in C}Hc$.
\end{lemma}
\begin{proof}
Let $X=\bigcap_{c\in C}Hc$. If $C\subseteq H$, then we are done. Thus, we may assume that $C\not\subseteq H$. Then, either $X=\emptyset$, or $X=Hc$ for some $c\in C\setminus H$ and $C\subseteq Hc$. On the other hand, by Lemma \ref{degrees}, $\deg_\G(x)=|C|$ if and only if $|x^{-1}H\cap C|=|C|$, that is, $C\subseteq x^{-1}H$ or equivalently $C\subseteq Hx$. Hence, the result follows.
\end{proof}
\begin{theorem}\label{semi-regular}
$\G$ is semi-regular if and only if one of the following conditions holds:
\begin{itemize}
\item[(i)]$C$ has the same number of elements in all left cosets of $H$ different from $H$, or
\item[(ii)]$C$ falls in a single right coset of $H$ different from $H$,
\end{itemize}
\end{theorem}
\begin{proof}
First suppose that $\G$ is semi-regular. Let $X=\bigcap_{c\in C}Hc$. Then, by previous lemma, $\deg_\G(x)=|C|$ if $x\in H\cup X$ and $\deg_\G(x)=\delta$ for all $x\in G\setminus H\cup X$, for some $\delta<|C|$. If $X=\emptyset$, then $|gH\cap C|=\delta$ for all $g\in G\setminus H$ and we get part (i). If $X\neq\emptyset$, then $C\subseteq Hc$ for some $c\in C\setminus H$ and we get part (ii). The converse is obvious by Lemma \ref{degrees}.
\end{proof}
\section{Connectivity}
In this section, we give necessary and sufficient conditions for connectivity of relative Cayley graphs and apply it to obtain sharp upper bounds for their diameters. In what follows, $N_\G^i(v)$ denotes the \textit{$i$th neighborhood} of a vertex $v$ of $\G$, for all $i\geq1$, that is, the set of all vertices $u$ of $\G$ for which there is a walk of length $i$ from $u$ to $v$. The following two results are obvious.
\begin{lemma}\label{isolatedvertex}
If $x\in G\setminus HC^*$, then $x$ is an isolated vertex.
\end{lemma}
\begin{corollary}\label{connectioncondition}
If $\G$ is a connected graph, then $G=HC^*$.
\end{corollary}
\begin{theorem}\label{connectivity}
The graph $\G$ is connected if and only if 
\[H=(H\cap gC)\gen{H\cap C}\gen{H\cap(C\setminus H)^2}\]
for some $g\in G\setminus H$.
\end{theorem}
\begin{proof}
Let $g\in G\setminus H$ and $N=H\cap gC$. Clearly, $N=N_\G(g)$ and $N(H\cap C)^i\subseteq N_\G^{i+1}(g)$ for all $i\geq0$. Thus 
\[M=\bigcup_{i=0}^\infty N(H\cap C)^i=N\left(\bigcup_{i=0}^\infty(H\cap C)^i\right)=N\gen{H\cap C}\]
belongs to the connected component containing $g$. Let $f$ be a function defined on the subsets of $H$ by $f(X)=H\cap X(C\setminus H)C$ for all $X\subseteq H$. Then, a simple verification shows that $\G$ is connected if and only if $G$ is the union of the following subsets
\[M,\ M(C\setminus H),\ f(M),\ f(M)(C\setminus H),\ f(f(M)),\ldots\]
Now, if $X$ is a subset of $H$, then 
\begin{align*}
f(X)&=H\cap X(C\setminus H)C\\
&=X(H\cap (C\setminus H)C)\\
&=X(H\cap (C\setminus H)^2),
\end{align*}
from which it follows that $f^i(X)=X(H\cap(C\setminus H)^2)^i$, where $f^0(X)=X$ and $f^{i+1}(X)=f(f^i(X))$ for all $i\geq0$. Thus 
\begin{align*}
\bigcup_{i=0}^\infty f^i(M)&=\bigcup_{i=0}^\infty M(H\cap(C\setminus H)^2)^i\\
&=M\left(\bigcup_{i=0}^\infty (H\cap(C\setminus H)^2)^i\right)\\
&=M\gen{H\cap(C\setminus H)^2}.
\end{align*}
In particular, $\G$ is connected if and only if 
\[G=M\gen{H\cap(C\setminus H)^2}\cup M\gen{H\cap(C\setminus H)^2}(C\setminus H),\]
which is equivalent to say that $H=M\gen{H\cap(C\setminus H)^2}$ and consequently
\[H=(H\cap gC)\gen{H\cap C}\gen{H\cap(C\setminus H)^2}.\]
The proof is complete.
\end{proof}
\begin{corollary}
If $H\cap C=\emptyset$, then $\G$ is connected if and only if $G=HC^*$ and 
\[H=\gen{H\cap C^2}.\]
\end{corollary}
\begin{remark}
A group $K$ is an \textit{$ABA$-group} if it has two proper subgroups $A$ and $B$ such that $K=ABA$. Utilizing this notion, if $H$ is not an $ABA$-group, then $\G$ is connected if and only if $G=HC^*$ and either $H\cap C$ or $H\cap (C\setminus H)^2$ generates $H$.
\end{remark}

Let $K$ be a group generated by a set $X$. The \textit{width} of $K$ with respect to $X$, denoted by $w(K,X)$, is the minimum non-negative integer $n$ such that
\[K=X^0\cup X\cup X^2\cup\cdots\cup X^n,\]
where $X^0=\{1\}$. If there is no such $n$, we define $w(K,X)$ to be $\infty$.
\begin{corollary}\label{diameter}
If $\G$ is connected, then
\[\diam(\G)\leq2+w(\gen{H\cap C},H\cap C)+2w(\gen{H\cap(C\setminus H)^2},H\cap(C\setminus H)^2).\]
In particular,
\[\diam(\G)\leq2+\frac{1}{2}|\gen{H\cap C}|+|\gen{H\cap(C\setminus H)^2}|\leq\frac{3}{2}|H|+2.\]
Also, 
\[\diam(\G)\leq|H|+2\]
if $H\cap C=\emptyset$ and
\[\diam(\G)\leq\frac{1}{2}|H|+2\]
if $H\cap(C\setminus H)^2=\{1\}$.
\end{corollary}
\begin{proof}
The first inequality follows directly from Theorem \ref{connectivity} by counting the number of steps required to reach to every vertex of $\G$ by starting from an arbitrary element. Also, the other inequalities follows from \cite{mfdg}. Note that $H=\gen{H\cap C}$ whenever $H\cap(C\setminus H)^2=\{1\}$ and $\G$ is connected. Also, $H=\gen{H\cap(C\setminus H)^2}$ whenever $H\cap C=\emptyset$ and $\G$ is connected.
\end{proof}
\begin{example}
\ 
\begin{itemize}
\item Let $G=\gen{a,b:a^{2n}=b^2=1,a^b=a^{-1}}$ be the dihedral group of order $4n$ and $H=\gen{a}$. If $C=\{a^{\pm1},b\}$, then it is easy to see that $\G=\Cay(G,H,C)$ is a corona $2n$-cycle, that is, a graph obtained from a cycle of length $2n$ by attaching a new pendant to each of its vertices. Clearly, $\diam(\G)=n+2=\frac{1}{2}|H|+2$. 

\item Let $G=\gen{a:a^{4mn}=1}$ be the cyclic group of order $4mn$ and $H=\gen{a^{2n}}$. If $C=\{a^{\pm1},\ldots,a^{\pm n}\}$, then it is easy to see that $\G=\Cay(G,H,C)$ is a bipartite graph with $\diam(\G)=2m+2=|H|+2$.
\end{itemize}
The above examples show that the last two bounds in Corollary \ref{diameter} are sharp.
\end{example}
\section{Numerical invariants}
In this section, we shall consider various numerical invariants of the relative Cayley graphs. In what follows $\omega(\G)$, $\alpha(\G)$, $\alpha'(\G)$, $\beta(\G)$, $\beta'(\G)$, $\chi(\G)$ and $\chi'(\G)$ denote the clique number, independence number, edge independence number, dominating number, edge dominating number, chromatic number and edge chromatic number of $\G$, respectively.
\begin{theorem}\label{cliquenumberupperbound}
$\omega(\G)\leq|H\cap C|+2$ and the equality holds if and only if $(H\cap C)^*\leq G$ and $H\cap C\subseteq cC$ for some $c\in C\setminus H$.
\end{theorem}
\begin{proof}
Let $X$ be a clique in $\G$. Clearly $|X\setminus H|\leq 1$. If $y\in X\cap H$, then for every $z\in(X\cap H)\setminus\{y\}$ there exist $c_z\in C$ such that $z=yc_z$. Thus $c_z=y^{-1}z\in H$ so that $c_z\in H\cap C$. Therefore $|X|\leq|H\cap C|+2$ and consequently $\omega(\G)\leq|H\cap C|+2$. 

Now, assume that $\omega(\G)=|H\cap C|+2$. Then there exist a clique $X$ with $|H\cap C|+2$ elements. Clearly, $X=\{x,y\}\cup y(H\cap C)$ for some $x\in G\setminus H$ and $y\in H$. Let $c_1,c_2\in H\cap C$. Since $X$ is complete, there exist $c\in C$ such that $yc_2=yc_1c$. Thus $c=c_1^{-1}c_2\in H\cap C$, from which it follows that $(H\cap C)^*$ is a subgroup of $G$. On the other hand, $x=yc$ for some $c\in C\setminus H$. Since $x$ is adjacent to $yc'$ for all $c'\in H\cap C$, it follows that $H\cap C\subseteq cC$. The converse is straightforward.
\end{proof}
\begin{corollary}
For every group $G$ and proper subgroup $H$, there exists a subset $C$ such that $\omega(\G)=|H\cap C|+2$.
\end{corollary}
\begin{proof}
Let $K$ be a subgroup of $H$ and $c\in G\setminus H$. If
\[C\supseteq (K\setminus\{1\})\cup c^{-1}(K\setminus\{1\})\cup(K\setminus\{1\})c\]
is any inverse closed subset of $G\setminus\{1\}$ such that $C\cap H=K\setminus\{1\}$, then the requirements of Theorem \ref{cliquenumberupperbound} are satisfied, and hence $\omega(\Cay(G,H,C))=|K|+1=|H\cap C|+2$, as required.
\end{proof}
\begin{definition}
Let $G$ be a group and $X$ be an arbitrary subset of $G$. Then 
\[\psi(X)=\max\{|H|:H\subseteq X^*,H\leq G\}\]
is the maximum order of subgroups contained in $X^*$.
\end{definition}
\begin{theorem}\label{cliquenumberbounds}
We have
\begin{itemize}
\item[(i)]$\omega(\G)\geq\psi(H\cap C)$. Moreover, $\omega(\G)\geq\psi(H\cap C)+1$ if there exists $c\in C\setminus H$ such that $K\subseteq cC$ for some subgroup $K$ of $G$ such that $K\subseteq(H\cap C)^*$ and $|K|=\psi(H\cap C)$.
\item[(ii)]If $C^3\subseteq C$, then $\omega(\G)\leq\psi(H\cap C)+1$.
\end{itemize}
\end{theorem}
\begin{proof}
(i) It is straightforward.

(ii) Let $X$ be a clique in $\G$ such that $|X\cap H|$ is maximum and 
\[K_x=\{1\}\cup\{c\in C:x\sim xc\in X\cap H\}\]
for all $x\in X\cap H$. Let $c_1,c_2$ be two nontrivial elements of $K_x$. Then $xc_1,xc_2\in X\cap H$ so that $xc_1\sim x c_2$ and so $x\sim xc_2c_1^{-1}$. Let $c\in C$ be any element such that $xc\in X\cap H$. Since $C^3\subseteq C$, $c'=c^{-1}c_2c_1^{-1}\in C$. Now, $c_2c_1^{-1}=cc'$, which implies that $xc\sim xc_2c_1^{-1}$. Therefore, $xc_2c_1^{-1}$ is adjacent to all elements of $xK_x$. If $c_2{c_1}^{-1}\notin K_x$, then $Y=(X\cap H)\cup\{xc_2c_1^{-1}\}$ induces a complete subgraph of $\G$ such that $|Y\cap H|=|X\cap H|+1$, which contradicts the maximality of $X\cap H$. Hence,	 $c_2c_1^{-1}\in K_x$, from which it follows that $K_x$ is a subgroup of $G$. Therefore, 
\[\omega(\G)\leq|X|+1=|K_x|+1\leq\psi(X)+1\leq\psi(H\cap C)+1,\]
as required.
\end{proof}
\begin{remark}
If $C^3\subseteq C$, then $C=Dc$ for some subgroup $D$ of $G$ and element $c\in C$ such that $c^2\in D$ and $D^c=D$. To see this, note that $C\subseteq C^3$. Thus $C^3=C$ and consequently $C^4=C^2$. Hence, $D=C^2$ is a subgroup of $G$. Since $C=DC$, we have $c'Dc^{-1}D\subseteq D$ for all $c,c'\in C$, which implies that $c'c^{-1}\in D$ or $c'\subseteq Dc$. Therefore, $C=Dc$. Moreover, $cDc^{-1}D\subseteq D$, which implies that $D^c\subseteq D$. Hence $D^c=D$.
\end{remark}
\begin{theorem}
We have
\begin{itemize}
\item[(i)]$\alpha (\G)=\beta'(\G)=|G\setminus H|$.
\item[(ii)]$\alpha'(\G)=\beta(\G)=|H|$.
\end{itemize}
\end{theorem}
\begin{proof}
Since $G\setminus H$ is independent and each element of $H$ is adjacent to some element of $G\setminus H$, one gets $G\setminus H$ is a maximal independent set of $\G$. Now, suppose that $I$ is an arbitrary independent set with maximum cardinality. We prove that $|I|=|G\setminus H|$. Clearly, $I=(I\cap H)\cup(I\cap (G\setminus H))$. Since $I$ is an independent set, $I\cap (I\cap H)C=\emptyset$ and by the same argument as before $(I\cap H)(C\setminus H)\cap(I\cap(G\setminus H))=\emptyset$. Therefore $I\cap (G\setminus H)\subseteq (G\setminus H)\setminus (I\cap H)(C\setminus H)$ so that 
\[|I\cap(G\setminus H)|\leq|G\setminus H|-|(I\cap H)(C\setminus H)|\leq|G\setminus H|-|I\cap H|.\]
Hence 
\[|I|=|I\cap H|+|I\cap(G\setminus H)|\leq|G\setminus H|.\]
On the other hand, $|I|\geq|G\setminus H|$, from which it follows that $|I|=|G\setminus H|$ and hence $\alpha(\G)=|G\setminus H|$. By \cite[Corollary 5.2.5]{rb-kr}, $\beta(\G)=|G|-\alpha(\G)=|H|$.

Since $G\setminus H$ is independent, it follows that $\alpha'(\G)\leq|H|$. On the other hand, if $c\in C\setminus H$, then $M=\{\{h,hc\}:h\in H\}$ is a matching with $|H|$ elements, which implies that $\alpha'(\G)\geq|H|$. Hence $\alpha'(\G)=|H|$ and consequently, by \cite[Theorem 5.3.3]{rb-kr}, $\beta'(\G)=|G|-\alpha'(\G)=|G\setminus H|$, as required.
\end{proof}
\begin{remark}
In general $G\setminus H$ is not the only independent set with maximum cardinality for if $C\setminus H$ is a singleton, then $A\cup B$ is an independent set of maximum cardinality, where $A$ is any subset of $H$ and $B=G\setminus AC^*$.
\end{remark}

Now, we study the chromatic number and edge chromatic number of the relative Cayley graphs. Note that by a well-known theorem of Vizing $\chi'(\G)$ is equal to either $\Delta(\G)$ or $\Delta (\G)+1$. In the former, $\G$ is said to be of class one and in the latter $\G$ is said to be of class two.
\begin{theorem}
$\G$ is a class one graph.
\end{theorem}
\begin{proof}
Let $c\in C\setminus H$ be a fixed element. We construct an edge coloring of $\G$ using $(C\setminus\{c\})\cup\{1\}$, from which the result follows. By Vizing's Theorem, $\chi'(\G')\leq|H\cap C|+1$. Hence, we may color the edges of $\G'$ using $(H\cap C)\cup\{1\}$. Now, let $\{h,hd\}$ be an edge of $\G$, which does not belong to $\G'$, where $h\in H$ and $d\in C\setminus H$. If the edges incident to $h$ in $\G'$ are colored with $H\cap C$, then color the edge $\{h,hd\}$ with $d$ if $d\neq c$ and with $1$ if $d=c$. Otherwise, there exists an edge incident to $h$ in $\G'$ colored with $1$. Then, there exists $e\in(H\cap C)\cup\{1\}$ such that the edges incident to $h$ in $\G'$ are colored with $((H\cap C)\cup\{1\})\setminus\{e\}$. Hence, we color $\{h,hd\}$ with $d$ if $d\neq c$ and with $e$ if $d=c$. It is easy to see that, the resulting edge coloring of $\G$ is proper, as required.
\end{proof}

It is easy to see that $\chi (\G')\leq\chi(\G)\leq\chi(\G')+1$. In the next example, we show that there is a group $G$ and a subgroup $H$ such that the chromatic number of their corresponding relative Cayley graph $\G$ equals $\chi(\G')$ or $\chi(\G')+1$, with respect to two different inversed closed subset $C$ of $G\setminus\{1\}$..
\begin{example}
Let $G=\gen{a,b:a^5=b^2=1,a^b=a^{-1}}$ be the dihedral group of order $10$ and $H=\gen{a}$.
\begin{itemize}
\item[(i)]If $C=\{a^{\pm1},b\}$, then $\chi(\G)=\chi(\G')=3$. (See Figure 1)
\item[(ii)]If $C=\{a^{\pm1},b,ab,a^{-1}b\}$, then $\chi(\G)=4$ and $\chi(\G')=3$. (See Figure 2)
\end{itemize} 
In Figures 1 and 2, doubled edges illustrate $\G'$ and all the edges illustrate $\G$.
\begin{center}
\begin{tabular}{ccc}
\begin{tikzpicture}
\node [circle,fill=green,inner sep=2pt] (A) at ({cos(72*0+18)},{sin(72*0+18)}) {};
\node [circle,fill=blue,inner sep=2pt] (B) at ({cos(72*1+18)},{sin(72*1+18)}) {};
\node [circle,fill=red!80,inner sep=2pt] (C) at ({cos(72*2+18)},{sin(72*2+18)}) {};
\node [circle,fill=green,inner sep=2pt] (D) at ({cos(72*3+18)},{sin(72*3+18)}) {};
\node [circle,fill=red!80,inner sep=2pt] (E) at ({cos(72*4+18)},{sin(72*4+18)}) {};
\node [circle,fill=red!80,inner sep=2pt] (AA) at ({2*cos(72*0+18)},{2*sin(72*0+18)}) {};
\node [circle,fill=red!80,inner sep=2pt] (BB) at ({2*cos(72*1+18)},{2*sin(72*1+18)}) {};
\node [circle,fill=green,inner sep=2pt] (CC) at ({2*cos(72*2+18)},{2*sin(72*2+18)}) {};
\node [circle,fill=red!80,inner sep=2pt] (DD) at ({2*cos(72*3+18)},{2*sin(72*3+18)}) {};
\node [circle,fill=green,inner sep=2pt] (EE) at ({2*cos(72*4+18)},{2*sin(72*4+18)}) {};

\draw [color=gray] (A)--(AA);
\draw [color=gray] (B)--(BB);
\draw [color=gray] (C)--(CC);
\draw [color=gray] (D)--(DD);
\draw [color=gray] (E)--(EE);
\draw [thick, double, color=gray] (A)--(B)--(C)--(D)--(E)--(A);
\end{tikzpicture}
&\hspace{1cm}&
\begin{tikzpicture}
\node [circle,fill=yellow,inner sep=2pt] (A) at ({cos(72*0+18)},{sin(72*0+18)}) {};
\node [circle,fill=yellow,inner sep=2pt] (B) at ({cos(72*1+18)},{sin(72*1+18)}) {};
\node [circle,fill=yellow,inner sep=2pt] (C) at ({cos(72*2+18)},{sin(72*2+18)}) {};
\node [circle,fill=yellow,inner sep=2pt] (D) at ({cos(72*3+18)},{sin(72*3+18)}) {};
\node [circle,fill=yellow,inner sep=2pt] (E) at ({cos(72*4+18)},{sin(72*4+18)}) {};
\node [circle,fill=green,inner sep=2pt] (AA) at ({2*cos(72*0+18)},{2*sin(72*0+18)}) {};
\node [circle,fill=blue,inner sep=2pt] (BB) at ({2*cos(72*1+18)},{2*sin(72*1+18)}) {};
\node [circle,fill=red!80,inner sep=2pt] (CC) at ({2*cos(72*2+18)},{2*sin(72*2+18)}) {};
\node [circle,fill=green,inner sep=2pt] (DD) at ({2*cos(72*3+18)},{2*sin(72*3+18)}) {};
\node [circle,fill=red!80,inner sep=2pt] (EE) at ({2*cos(72*4+18)},{2*sin(72*4+18)}) {};

\draw [color=gray] (A)--(AA);
\draw [color=gray] (B)--(BB);
\draw [color=gray] (C)--(CC);
\draw [color=gray] (D)--(DD);
\draw [color=gray] (E)--(EE);
\draw [color=gray] (AA)--(B)--(CC)--(D)--(EE)--(A)--(BB)--(C)--(DD)--(E)--(AA);
\draw [thick, double, color=gray] (AA)--(BB)--(CC)--(DD)--(EE)--(AA);
\end{tikzpicture}\\
\mbox{Figure 1}&&\mbox{Figure 2}
\end{tabular}
\end{center}
\end{example}

The following theorem gives a sharp upper bound for the chromatic number of relative Cayley graphs.
\begin{theorem}
$\chi(\G)\leq|H\cap C|+2$ and the equality follows if and only if 
\begin{itemize}
\item[(i)]$gH\cup H\setminus\{1\}\subseteq C$ for some $g\in G\setminus H$, or
\item[(ii)]$H=\gen{h}$ is cyclic and for each partition $H=X_1\cup X_2\cup X_3$ of $H$ such that $X_ih\cap X_i=\emptyset$ for $i=1,2,3$, there exists $g\in G\setminus H$ such that $C\cap gX_i\neq\emptyset$ for $i=1,2,3$.
\end{itemize}
\end{theorem}
\begin{proof}
It is well-known from spectral graph theory that $\chi(\G')\leq\Delta(\G')+1=|H\cap C|+1$ (see \cite[Lemma 3.17 and Theorem 3.18]{rbb}). Now, since $G\setminus H$ is an independent set in $\G$, it follows that $\chi(\G)\leq|H\cap C|+2$. In the sequel assume that $\chi(\G)=|H\cap C|+2$. Then $\chi(\G')=|H\cap C|+1$ and by Brook's theorem (see \cite{rlb}), either $\G'$ is complete or $\G'$ is an odd cycle. If $\G'$ is complete, then $H\setminus\{1\}\subseteq C$. On the other hand, $\chi(\G)=\chi(\G')+1$, which implies that there exists a vertex $g\in G\setminus H$ adjacent to all elements of $H$. Hence $H\cap gC=N_\G(g)=H$ so that $g^{-1}H\subseteq C$ and part (i) follows. Finally, suppose that $\G'$ is an odd cycle. Then for every coloring of $H$ into three colors with color classes $X_1$, $X_2$ ad $X_2$ there must exists an element $g\in G\setminus H$ adjacent to some elements of $X_i$ for $i=1,2,3$, from which part (ii) follows. The converse is obvious.
\end{proof}
\section{Forbidden subgraphs}
In this section, we shall study forbidden structures in relative Cayley graphs. Recall that a graph is \textit{$S$-free} if is has no induced subgraphs isomorphic to $S$. In what follows, the structure of claw-free graphs as well as $C_n$-free graphs will be considered, where a \textit{claw} is the star graph $K_{1,3}$. In addition, the structure of relative Cayley trees with be classified.
\begin{theorem}
$\G$ is claw-free if and only if one of the following conditions hold:
\begin{itemize}
\item[(i)]$|C|\leq2$,
\item[(ii)]$C=\{a,b,ab^{-1},ba^{-1}\}$ such that $a,b\in G\setminus H$ belong to the same right coset of $H$,
\item[(iii)]$C\setminus H=\{c\}$ and $(H\cap C)^*\leq G$ for some involution $c\in G\setminus H$, or
\item[(iv)]$C\subseteq H$ and $1\notin(H\setminus C^*)^3$.
\end{itemize}
\end{theorem}
\begin{proof}
First suppose that $\G$ is claw-free. If $|C|\leq2$, then we are done. Hence, we assume that $|C|>2$. Since $h(C\setminus H)\subseteq N_\G(h)$ is independent for all $h\in H$, we must have $|C\setminus H|\leq2$. Then $H\cap C\neq\emptyset$ and we have the following three cases:

Case 1. $|C\setminus H|=2$. Let $C\setminus H=\{a,b\}$. If $c\in H\cap C$ such that $c\{a,b\}\cap\{a,b\}=\emptyset$, then $\{1,a,b,c\}$ induces a claw, which is a contradiction. Thus $c\{a,b\}\cap\{a,b\}\neq\emptyset$ for all $c\in H\cap C$, from which it follows that $H\cap C=\{ab^{-1},ba^{-1}\}$, as required. Clearly, $a,b$ belong to the same right cosets of $H$.

Case 2. $|C\setminus H|=1$. Let $C\setminus H=\{c\}$. Then $c^2=1$. If $a\in H\cap C$ is adjacent to $c$, then there exists $b\in C$ such that $ab=c$. But then $b\in C\setminus H=\{c\}$ and we obtain $a=1$, which is a contradiction. Hence $N_\G(c)\cap(H\cap C)=\emptyset$. Now, since $\G$ is claw-free and $N_\G(1)=C$, the subgraph of $\G$ induced by $H\cap C$ is complete, that is, $(H\cap C)^*\leq H$.

Case 3. $|C\setminus H|=0$. Then $C\subseteq H$. Since $\G'$ is claw-free it follows that $\G'^c$ is triangle free, from which it follows that $1\notin(H\setminus C^*)^3$.

The converse is straightforward.
\end{proof}
\begin{lemma}\label{forest}
The following conditions are equivalent:
\begin{itemize}
\item[(i)]$\G$ is a forest,
\item[(ii)]$\G$ is both triangle-free and square-free, and $\Cay(H,H\cap(C\cup C^2\setminus\{1\}))$ is a forest.
\item[(iii)]$H\cap C^2=\{1\}$ and either $H\cap C=\emptyset$ or $H\cap C=\{c\}$ for some involution $c$.
\end{itemize}
\end{lemma}
\begin{proof}
(i)$\Rightarrow$(ii).  Clearly, $\G$ is triangle-free and square-free, which enables us to construct a graph $\G^*$ obtained from $\G$ by removing $G\setminus H$ from the vertices and adding edges of the form $\{h,hc_1c_2\}$ for all $h\in H$ and $c_1,c_2\in C\setminus H$ such that $c_1c_2\in H$. Then $\G^*\cong\Cay(H,H\cap(C\cup C^2\setminus\{1\}))$. Clearly, there is a one-to-one correspondence between cycles of $\G$ and cycles of $\G^*$. Hence $\G^*$ is a forest.

(ii)$\Rightarrow$(iii). Let $D=H\cap(C\cup C^2\setminus\{1\})$. If $D=\emptyset$, then the result follows. Thus, we may assume $D\neq\emptyset$. If $c\in D$, then $\G$ contains a cycle whose vertices are elements of $\gen{c}$ unless $c$ is an involution. Hence, $D$ contains only involutions. If $|D|\geq2$ and $a,b\in D$ are distinct involutions, then $\G$ has a cycle of length $2|ab|$ whose vertices are elements of $\gen{a,b}$, which is a contradiction. Thus $D=\{c\}$, where $c$ is an involution. If $H\cap C^2\neq\{1\}$, then $c\in H\cap C^2$ so that $c=c_1c_2$ for some $c_1,c_2\in C$. But then $\{1,c_1,c,cc_1\}$ induces a $4$-cycle in $\G$, which is a contradiction. Thus $H\cap C^2=\{1\}$ so that $H\cap C=\{c\}$.

(iii)$\Rightarrow$(i). A simple verification shows that $\G$ is a union of star or double-starts, that is, $\G$ is a forest.
\end{proof}
\begin{corollary}
$\G$ is a tree if and only if $H=1$ and $C=G\setminus\{1\}$, or $H=\gen{a}\cong\Z_2$ and $C=Da$, where $G=DH$ and $D\cap H=\{1\}$.
\end{corollary}
\begin{lemma}\label{triangle-free}
$\G$ is triangle-free if and only if $H\cap C\cap C^2=\emptyset$.
\end{lemma}
\begin{proof}
If $\G$ has a triangle, then it contains a triangle of the form $\{1,c_1,c_1c_2\}$, where $c_1\in H\cap C$ and $c_2\in C$. Since $1$ and $c_1c_2$ are adjacent, there exists $c\in C$ such that $c_1c_2=1\cdot c=c$. Hence, $c_1=cc_2^{-1}\in H\cap C\cap C^2$ so that $H\cap C\cap C^2\neq\emptyset$. Conversely, if $H\cap C\cap C^2\neq\emptyset$, then we may simply find a triangle. Hence the result follows.
\end{proof}
\begin{lemma}\label{square-free}
$\G$ has no squares as subgraph if and only if the following conditions hold:
\begin{itemize}
\item[(1)]$\G'$ has no squares as subgraph,
\item[(2)]$(H\cap C)^2\cap(C\setminus H)^2=\{1\}$,
\item[(3)]$\sum_{c\in C\setminus H}\deg_\G(c)=|H\cap(C\setminus H)^2|+|C\setminus H|$.
\end{itemize}
\end{lemma}
\begin{proof}
Suppose $\G$ has an square $S$. Clearly, $|S\cap H|\geq2$ and we may assume without loss of generality that $1\in S$. We have three cases:

Case 1. $|S\cap H|=4$. Then $\G'$ has the square $S$ as a subgraph.

Case 2. $|S\cap H|=3$. Then $S\cap H=\{1,h,k\}$ and $S\setminus H=\{x\}$ for some $h,k\in H\cap C$ and $x\in G\setminus H$, and we may assume that $x$ is adjacent to both $h$ and $k$. So, there exists $c_h,c_k\in C\setminus H$ such that $x=hc_h=kc_k$. Hence $1\neq h^{-1}k=c_hc_k^{-1}\in(H\cap C)^2\cap(C\setminus H)^2$.

Case 3. $|S\cap H|=2$. Then $S\cap H=\{1,h\}$ and $S\setminus H=\{c_1,c_2\}$ for some $h\in H$ and $c_1,c_2\in C\setminus H$. Let 
\[X=\{(c_1,c_2)\in (C\setminus H)\times(C\setminus H):c_1c_2\in H\}.\]
Since $h$ is adjacent to $c_1,c_2$, there exist $c'_1,c'_2\in C\setminus H$ such that $hc'_1=c_1$ and $hc'_2=c_2$. Hence $c_1c_1'^{-1}=c_2c_2'^{-1}=h$, which implies that $|X|>|H\cap(G\setminus H)^2|+|C\setminus H|$.

Utilizing the cases 1, 2 and 3, the result follows immediately.
\end{proof}

Even thought the structure of $\{C_n\}_{n\geq3}$-free graphs (forests) can be determined simply (see Theorem \ref{forest}), but the structure of $\{C_{2n+1}\}_{n\geq1}$-free graphs (bipartite graphs) seems to be a difficult problem. The following result provides a partial answer to this problem.
\begin{lemma}\label{bipartite}
If $H\cap C=\emptyset$, then $\G$ is a bipartite graph.
\end{lemma}
\begin{proof}
By definition $G\setminus H$ is an independent set. Now, if $x,y\in H$ are adjacent, then $x^{-1}y\in H\cap C$, which is a contradiction. Hence, $H$ is an independent set and consequently $\G$ is bipartite.
\end{proof}
\begin{example}
Let $G=\gen{a,b:a^4=b^2=1,a^b=a^{-1}}$ be the dihedral group of order $8$. If $H=\gen{a^2,b}$ and $C=\{a,a^{-1},b\}$, then $H\cap C\neq\emptyset$ but $\G$ is bipartite as illustrated in Figure 3.
\end{example}
\begin{center}
\begin{tikzpicture}[scale=0.75]
\node [circle,fill=black,inner sep=2pt,label=left:\tiny{$a^{-1}b$}] (A) at (0,0) {};
\node [circle,fill=black,inner sep=2pt,label=left:\tiny{$ab$}] (B) at (0,1) {};
\node [circle,fill=black,inner sep=2pt,label=left:\tiny{$a^2$}] (C) at (0,2) {};
\node [circle,fill=black,inner sep=2pt,label=left:\tiny{$1$}] (D) at (0,3) {};
\node [circle,fill=black,inner sep=2pt,label=right:\tiny{$a^2b$}] (E) at (2,0) {};
\node [circle,fill=black,inner sep=2pt,label=right:\tiny{$b$}] (F) at (2,1) {};
\node [circle,fill=black,inner sep=2pt,label=right:\tiny{$a^{-1}$}] (G) at (2,2) {};
\node [circle,fill=black,inner sep=2pt,label=right:\tiny{$a$}] (H) at (2,3) {};

\draw [color=gray] (A)--(E)--(B)--(F)--(A);
\draw [color=gray] (C)--(G)--(D)--(H)--(C);
\draw [color=gray] (C)--(E);
\draw [color=gray] (D)--(F);
\end{tikzpicture}\\
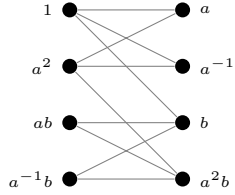
Figure 3.
\end{center}
\begin{problem}
Is there any simple characterization of bipartite relative Cayley graphs $\Cay(G,H,C)$ in terms of $G$, $H$ and $C$?
\end{problem}

\end{document}